\documentclass[12pt]{amsart}

\usepackage{a4wide}

\usepackage{amsmath}
\usepackage{amssymb}
\usepackage{amsfonts}

\newtheorem{thm}{Theorem}[section]
\newtheorem{proposition}[thm]{Proposition}
\newtheorem{cor}[thm]{Corollary}
\newtheorem{lem}[thm]{Lemma}
\newtheorem{dfn}[thm]{Definition}
\newtheorem{fct}[thm]{Fact}

\newtheorem{rmk}[thm]{Remark}

\newtheorem{exa}[thm]{Example}

=22cm\headheight=0cm\topskip=0cm\topmargin=0cm

\newcommand{\M}{\sf M}

\def\dotminus{\mathbin{\ooalign{\hss\raise1ex\hbox{.}\hss\cr
  \mathsurround=0pt$-$}}}

\def\Ind#1#2{#1\setbox0=\hbox{$#1x$}\kern\wd0\hbox to 0pt{\hss$#1\mid$\hss}
\lower.9\ht0\hbox to 0pt{\hss$#1\smile$\hss}\kern\wd0}

\def\notind#1#2{#1\setbox0=\hbox{$#1x$}\kern\wd0\hbox to 0pt{\mathchardef
\nn=12854\hss$#1\nn$\kern1.4\wd0\hss}\hbox to
0pt{\hss$#1\mid$\hss}\lower.9\ht0 \hbox to
0pt{\hss$#1\smile$\hss}\kern\wd0}

\begin{document}
\def\dis{\displaystyle}

\begin{center}
{ {\sc  Stability, NIP, and NSOP; Model Theoretic Properties of
Formulas \\ via Topological Properties of Function Spaces }}
\vspace{10mm}

{ {\sc Karim Khanaki}} \vspace{3mm}

{\footnotesize
  Department of science, Arak University of Technology,
 \\
P.O. Box 38135-1177, Arak, Iran; e-mail: khanaki@arakut.ac.ir
\\ \bigskip School of Mathematics,
Institute for Research in Fundamental Sciences (IPM), \\ P.O. Box
19395-5746, Tehran, Iran;
 e-mail: khanaki@ipm.ir} \vspace{5mm}
\end{center}

{\sc Abstract.}
{\small    We study and characterize  stability, NIP and NSOP in
terms of topological and measure theoretical properties of
classes of functions. We study a measure theoretic property,
`Talagrand's stability', and explain the relationship between this
property and NIP in continuous logic. Using a result of Bourgain,
Fremlin and Talagrand, we prove the `almost definability'  and
`Baire~1 definability' of coheirs assuming NIP. We show that a
formula $\phi(x,y)$ has the strict order property if and only if
there is a convergent sequence of continuous functions on the
space of $\phi$-types such that its limit is not continuous. We
deduce from this  a theorem of Shelah and point out the
correspondence between this theorem and the Eberlein-\v{S}mulian
theorem.}

\medskip

{\small{\sc Keywords}: Talagrand's stability, 
independence property, coheir, strict order property,
continuous~logic, relative weak compactness, angelic space.}

AMS subject classification: 03C45, 03C90, 46E15, 46A50

\noindent\hrulefill

{\small \tableofcontents


\noindent\hrulefill

\newpage

\section{Introduction} \label{1}
In \cite{Ros} Rosenthal introduced the independence property for
families of real-valued functions and used this property for
proving his celebrated $l^1$ theorem: a Banach space is either
`good' (every bounded sequence has a weak-Cauchy subsequence) or
`bad' (contains an isomorphic copy of $l^1$). After this and
another work of Rosenthal \cite{Ros2},
 Bourgain, Fremlin and Talagrand \cite{BFT} found some
topological and measure theoretical criteria for the independence
property and proved that the space of functions of the first
Baire class on a Polish space is angelic; a topological notion for
which the terminology was introduced by Fremlin. This theorem
asserts that a set of continuous functions on a Polish space is
either `good' (its closure is precisely the set of limits of its
sequences) or `bad' (its closure contains non-measurable
functions). In fact these dichotomies correspond to the NIP/IP
dichotomy in continuous logic; see Fact~\ref{BFT} below.

In this paper we propose a generalization of Shelah's dividing
lines for classification of first order theories which deals with
real-valued formulas instead of 0-1~valued formulas. The
principal aim of this paper is to study and characterize some
model theoretic properties of formulas, such as OP, IP and SOP, in
terms of topological and measure theoretical properties of
function spaces. This study enables us to obtain new results and
to reach a better understanding of the known results.

Let us give the background and our own point of view. In Shelah's
stability theory, the set-theoretic criteria lead to ranks or
combinatorial properties of a particular formula. There are known
interactions between some of these combinatorial properties and
some topological properties of function spaces. As an example, a
formula $\phi(x,y)$ has the order property (OP) if there exist
$a_ib_i,i<\omega$ such that $\phi(a_i,b_j)$ holds if and only if
$i<j$. One can assume that $\phi$ is a 0-1 valued function, such
that $\phi(a,b)=1$ iff $\phi(a,b)$ holds. Then $\phi$ has the
order property iff there exist $a_i,b_j$ such that
$\lim_i\lim_j\phi(a_i,b_j)=1\neq 0=\lim_j\lim_i\phi(a_i,b_j)$.
Thus failure of the order property, or stability, is equivalent
to the requirement that the double limits $\lim_i\lim_j\phi$ and
$\lim_j\lim_i\phi$ be the same.  Using a crucial result due to
Eberlein and Grothendieck, the latter is a topological property
of a family of functions; see Fact~\ref{Fact2} below. Similarly,
using the result of Bourgain, Fremlin and Talagrand mentioned
above, one can obtain some topological and measure theoretical
characterizations of NIP formulas. Therefore, it seems reasonable
that one studies real-valued formulas and hopes to obtain new
classes of functions (formulas) and develop a sharper stability
theory by making use of topological properties of function spaces
instead of only combinatorial properties of formulas.

In this paper  (except in Section 4) we work in continuous logic
which is an extension of classical first order logic; thus our
results hold in the latter case.


The following is a summary of the main results of this paper:
 Propositions \ref{SCP->NSOP}, \ref{NSOP=SCP}, \ref{Shelah=Eberlein}, Theorems
\ref{NIP-compactness}  and \ref{almost-dfn} are new results. Also,
Definitions \ref{NIP-formula} and \ref{universal dfn}   are new.
Propositions \ref{NIP-almost}, \ref{NIP-dfn}, \ref{Keisler-NIP}
and Theorem \ref{Baire-dfn} have not previously been published
but are essentially just translations from functional analysis.
Section~3 focuses on  NIP in the framework of continuous logic
and Section~4 focuses on SOP in classical model theory.


This is not the end of the story if one defines a notion of
non-forking extension in NIP theories such that it satisfies
symmetry and transitivity. Moreover, one can study sensitive
families of functions, dynamical systems and chaotic maps and
their connections with stability theory. We will study them in a
future work.

It is worth recalling another line of research. After the
preparation of the first version of this paper, we came to know
that simultaneously in \cite{Iba14} and \cite{S2} the
relationship between NIP and Rosenthal's dichotomy was noticed in
the contexts of $\aleph_0$-categorical structures in continuous
logic and classical first order setting, respectively.
Independently,  the relationship between NIP in integral logic and
Talagrand's stability was studied in \cite{K}.

This paper is organized as follows: In the second section, we
briefly review continuous logic and stability. In the third
section, we study Talagrand's stability and its relationship with
NIP in logic, and give some characterizations of NIP in terms of
measure and topology. The result of Bourgain, Fremlin and
Talagrand  is  used in this section for proving of definability of
coheirs in NIP theories. In the fourth section, we study the SOP
and point out the correspondence between Shelah's theorem and the
Eberlein-\v{S}mulian theorem.

\bigskip\noindent
{\bf Acknowledgements.} I am very much indebted to Professor
David H. Fremlin for his kindness and his helpful comments. I am
grateful to M\'{a}rton Elekes for valuable comments and
observations, particularly Example~\ref{exa} below. I thank the
anonymous referees for their detailed suggestions and
corrections; they helped to improve significantly the exposition
of this paper.

 I would like to thank the Institute for Basic Sciences (IPM),
Tehran, Iran. Research partially supported by IPM grant 93030032.}

\noindent\hrulefill

\section{Continuous Logic} \label{2}
In this section we give a brief review of continuous logic from
\cite{BU} and \cite{BBHU}. Results stated without proof can be
found there. The reader who is familiar with continuous logic can
skip this section.

\subsection{Syntax and semantics} \label{syntax}
A {\em language} is a set $L$ consisting of constant symbols and
function/relation symbols of various arities. To   each relation
symbol $R$ is assigned a bound $\flat_R\in[0,\infty)$ and we
assume that its interpretations is bounded by $\flat_R$. It is
always assumed that $L$ contains the metric symbol $d$ and
$\flat_d=1$. We use $\mathbb{R}$ as value space and its common
operations $+,\times$ and scalar products as connectives.
Moreover to each relation symbol $R$ (function symbol $F$) is
assigned a modulus of uniform continuity $\Delta_R$ ($\Delta_F$).
We also use the symbols `sup' and `inf' as quantifiers.

Let $L$ be a language. {\em $L$-terms} and their bound are
inductively define as follows:

\begin{itemize}
  \item Constant symbols and variables are terms.
  \item If $F$ is a $n$-ary function symbol and $t_1,
\ldots,t_n$ are terms, then $F(t_1,\ldots, t_n)$ is a term.

All $L$-terms are constructed in this way.

\end{itemize}

\begin{dfn} $L$-formulas and their  bounds are inductively defined as follows:
\begin{itemize}
  \item Every $r\in\mathbb{R}$ is an atomic formula with
bound $|r|$.
  \item If $R$ is a $n$-ary relation symbol and $t_1,\ldots,t_n$ are terms,
$R(t_1,\ldots,t_n)$ is an atomic formula with bound $\flat_R$.
\item If $\phi,\psi$ are formula and $r\in\mathbb{R}$ then
$\phi+\psi,\phi\times\psi$ and $r\phi$ are formulas  with bound
resp $\flat_\phi+\flat_\psi, \flat_\phi\flat_\psi, |r|\flat_\phi$.
\item If $\phi$ is a formula and $x$ is a variable, $\sup_x\phi$ and $\inf_x\phi$ are formulas with the same
 bound as $\phi$.
\end{itemize}
\end{dfn}

\begin{dfn}  A {\em prestructure} in $L$ is pseudo-metric space $(M, d)$ equipped with:
\begin{itemize}
  \item for each constant symbol $c\in L$, an element $c^M\in M$
  \item for each $n$-ary function symbol $F$ a function $F^M : M^n\to M$ such that
 $$d_n^M(\bar x,\bar y)\leqslant\Delta_F(\epsilon) \Longrightarrow d^M(F^M(\bar x),F^M(\bar y))\leqslant\epsilon$$
  \item for each $n$-ary relation symbol $R$ a function  $R^M : M^n\to [-\flat_R,\flat_R]$ such that
   $$d_n^M(\bar x,\bar y)\leqslant\Delta_R(\epsilon) \Longrightarrow |R^M(\bar x)-R^M(\bar
y)|\leqslant\epsilon.$$
\end{itemize}
\end{dfn}

If $M$ is a prestructure, for each formula $\phi(\bar x)$ and
$\bar a\in M$, $\phi^M(\bar a)$ is defined inductively starting
from atomic formulas. In particular, $(\sup_y \phi)^M(\bar
a)=\sup_{b\in M}\phi^M(\bar a, b)$. Similarly for $\inf_y\phi$.

\begin{proposition} Let $M$ be an $L$-prestructure and $\phi(\bar x)$ a formula with $|\bar x|=n$. Then
$\phi^M(\bar x)$  is a  real-valued function on $M^n$ with a
modulus of uniform continuity $\Delta_\phi$ and $|\phi^M(\bar
a)|\leqslant\flat_\phi$ for every $\bar a$.
\end{proposition}

Interesting prestructures are those which are {\em complete}
metric spaces. They are called {\em $L$-structures}. Every
prestructure can be easily transformed to a complete
$L$-structure by first taking the quotient metric and then
completing the resulting metric space. By uniform continuity,
interpretations of function and relation symbols induce
well-defined function and relations on the resulting metric space.

\subsection{Compactness, types, stability} \label{compactness}
Let $L$ be a  language. An expression of the form
$\phi\leqslant\psi$, where $\phi,\psi$ are formulas, is called a
{\em condition}. The equality $\phi=\psi$  is called a condition
again. These conditions are called closed if $\phi,\psi$ are
sentences. A {\em theory} is a set of closed conditions. The
notion $M\models T$ is defined in the obvious way. $M$ is then
called a model of $T$. A theory is {\em satisfiable} if has a
model.

An ultraproduct construction can be defined. The most important
application of this construction in logic is to prove the \L
o\'{s} theorem and to deduce the compactness theorem.

\begin{thm}[Compactness Theorem] Let $T$ be an $L$-theory and $\mathcal{C}$ a class of
$L$-structures. Suppose that $T$ is finitely satisfiable in
$\mathcal{C}$. Then there exists an ultraproduct of structures
from $\mathcal{C}$ that is a model of $T$.
\end{thm}

There are intrinsic  connections between some concepts from
functional analysis and  continuous logic. For example, types are
well known mathematical objects, {\em Riesz homomorphisms}. To
illustrate this, there are two options; Gelfand representation of
$C^*$-algebras, and Kakutani representation of $M$-spaces. We work
in a real-valued logic, so we use the latter.

Suppose that $L$ is an arbitrary language.  Let $M$ be an
$L$-structure, $A\subseteq M$ and $T_A=Th({M}, a)_{a\in A}$. Let
$p(x)$ be a set of $L(A)$-conditions in free variable $x$. We
shall say that $p(x)$ is a {\em type  over} $A$ if $p(x)\cup T_A$
is satisfiable. A {\em complete type over} $A$ is a maximal type
over $A$. The collection of all such types over $A$ is denoted by
$S^{M}(A)$, or simply by $S(A)$ if the context makes the theory
$T_A$ clear. The {\em type of $a$ in $M$ over $A$}, denoted by
$\text{tp}^{M}(a/A)$, is the set of all $L(A)$-conditions
satisfied in $M$ by $a$. If $\phi(x,y)$ is a formula, a {\em
$\phi$-type} over $M$ is a maximal consistent set of formulas of
the form $\phi(x,a)\geqslant r$, for $a\in M$ and
$r\in\mathbb{R}$. The set of $\phi$-types over $M$ is denoted by
$S_\phi(M)$. The definition of a $\phi$-type over a set $A$ which
is not a model needs a few more steps (see Definition~6.6 in
\cite{BU}).

We now give a characterization of complete types in terms of
functional analysis. Let $\mathcal{L}_A$ be the family of all
interpretations $\phi^{M}$ in $M$ where $\phi$ is an
$L(A)$-formula with a free variable $x$. Then $\mathcal{L}_A$ is
an Archimedean Riesz space of measurable functions on $M$ (see
\cite{Fremlin3}). Let $\sigma_A({M})$ be the set of Riesz
homomorphisms $I: {\mathcal L}_A\to \mathbb{R}$ such that
$I(\textbf{1}) = 1$, where $\textbf{1}$ is the constant $1$
function on $M$. The set $\sigma_A({M})$ is   called the {\em
spectrum} of $T_A$. Note that $\sigma_A({M})$ is a weak* compact
subset of the dual space $\mathcal{L}_A^*$  of $\mathcal{L}_A$.
The next proposition shows that a complete type can be coded by a
Riesz homomorphism and gives a characterization of complete
types. In fact, by the  Kakutani representation theorem, the map
$S^{M}(A)\to\sigma_A({M})$, defined by $p\mapsto I_p$ where
$I_p(\phi^M)=r$ if $\phi(x) = r$ is in $p$, is a bijection.  By
adapting the proof of Proposition~5.6 of \cite{K}, one can show
that:

\begin{proposition} \label{key}
Suppose that $M$, $A$ and $T_A$ are as above.
\begin{itemize}
             \item [{\em (i)}] The map $S^{M}(A)\to\sigma_A({M})$  defined by $p\mapsto I_p$ is bijective.
             \item [{\em (ii)}] A set $p$ of $L(A)$-conditions is an element of $S^{M}(A)$ if and only if there is an elementary
extension $N$ of $M$ and $a\in N$ such that
$p=\text{tp}^{N}(a/A)$.
\end{itemize}
\end{proposition}

We equip $S^{M}(A)=\sigma_A({M})$  with the related topology
induced from $\mathcal{L}_A^*$. Therefore, $S^{M}(A)$ is a compact
and Hausdorff space. For any complete type $p$ and formula
$\phi$, we let $\phi(p)=I_p(\phi^{M})$. It is easy to verify that
the topology on $S^{M}(A)$ is the weakest topology in which all
the functions $p\mapsto \phi(p)$ are continuous. This topology is
sometimes called the {\em logic topology}. The same things are
true for $S_\phi(M)$.

\begin{dfn} 
\label{stab2}  A formula $\phi(x,y)$ is called {\em stable in a
structure} $M$ if there are no $\epsilon>0$ and infinite sequences
$a_n, b_n \in M$ such that for all $i<j$: $|\phi(a_i,b_j) -
\phi(a_j,b_i)| \geqslant \epsilon$.  A formula $\phi$ is {\em
stable in a theory} $T$ if it is stable in every model of $T$. If
$\phi$ is not stable in $M$ we say that it has the {\em order
property} (or short the OP). Similarly, $\phi$ has the OP in $T$
if it is not stable in some model of $T$.
\end{dfn}

It is easy to verify that $\phi(x,y)$ is stable in $M$ if whenever
$a_n,b_m\in M$ form two sequences we have
$$\lim_n\lim_m\phi(a_n,b_m)=\lim_m\lim_n\phi(a_n,b_m),$$
provided both limits exist.

\begin{lem} \label{stable formula}
  Let $\phi(x,y)$ be a formula.
  Then the following are equivalent:
  \begin{itemize}
  \item [{\em (i)}] The formula $\phi$ is stable.
  \item [{\em (ii)}] There are no distinct real numbers $r,s$ and
    infinite sequence $(a_ib_i\colon i < \omega)$ such that
    $\phi(a_i,b_j)=r$ for $i<j$ and  $\phi(a_i,b_j)=s$ for $i\geq j$.
  \end{itemize}
\end{lem}

By the following result,  stability of a formula $\phi(x,y)$ is
equivalent to the family of functions being  relatively weakly
compact.  In everything that follows, if $X$ is a topological
space then $C_b(X)$ denotes the Banach space of bounded
real-valued functions on $X$, equipped with the supremum norm. A
subset $A\subseteq C_b(X)$ is relatively weakly compact if it has
compact closure in the weak topology on $C_b(X)$. If $X$ is a
compact space, then we write $C(X)$ instead of $C_b(X)$.

\begin{fct}[\cite{Fremlin4}, Proposition~462E] \label{Grothendieck-lemma}
Let $X$ be a  compact topological space, and $A$ a subset of
$C(X)$. Then $A$ is weakly compact in $C(X)$ iff it is
norm-bounded and pointwise compact.
\end{fct}

In \cite{Gro}, Grothendieck says that the following is based on
an idea of Eberlein. (In \cite{Pillay-Grothendieck}, Pillay
correctly pointed out this.)

\begin{fct}[Eberlein-Grothendieck criterion, \cite{Gro}, Th\'{e}or\`{e}me~6]    \label{Criterion} 
 Let $X$ be an arbitrary topological space, $X_0\subseteq X$ a
dense subset. Then the following are equivalent for a subset
$A\subseteq C_b(X)$:
\begin{itemize}
             \item [{\em (i)}]  The set $A$ is relatively weakly compact in $C_b(X)$.
             \item [{\em (ii)}] The set
$A$ is bounded, and for any sequences
$\{f_n\}_{1}^\infty\subseteq A$ and $\{x_n\}_{1}^\infty\subseteq
X_0$, we have $$\lim_n \lim_m f_n(x_m) =\lim_m \lim_n
f_n(x_m),$$  whenever both limits exist.
\end{itemize}
\end{fct}

The following is a model-theoretic version of the
Eberlein-Grothendieck criterion, as pointed  out  by Ben~Yaacov in
\cite{Ben-Gro}  (see Fact 2, the discussion before Theorem 3 and
Theorem 5 therein).

\begin{cor} \label{Fact2} 
Let $M$ be a structure and $\phi(x,y)$ a formula. Then the
following are equivalent:
\begin{itemize}
             \item [{\em (i)}] $\phi(x,y)$ is stable in $M$.
             \item [{\em (ii)}]  The set $A=\{\phi(x,b):S_x(M)\to \mathbb{R}~|b\in M\}$ is relatively weakly compact in $C(S_x(M))$.
\end{itemize}
\end{cor}

\noindent\hrulefill
\section{NIP} \label{4}

In this section we study Talagrand's stability and its
relationship to NIP in continuous logic. Then, we give some
characterizations of NIP in terms of topology and measure, and
deduce various forms of definability of coheirs for NIP models.

\subsection{Independent family of functions}
In \cite{Ros} Rosenthal introduced the independence property for
families of real-valued functions and used it for proving his
dichotomy. As we will see shortly, this notion corresponds to a
generalization of the IP for real-valued formulas.

\begin{dfn}[\cite{GM}, Definition~2.8] \label{NIP-family}  
 A family $F$ of real-valued functions on a set $X$ is said to
be {\em independent}
 (or has the {\em independence property}, short IP) if there
exist real numbers $s<r$ and a sequence $f_n\in F$ such that for
each $k\geqslant 1$ and for each $I\subseteq\{1,\ldots,k\}$,
there is $x\in X$ with $f_i(x)\leqslant s$ for $i\in I$ and
$f_i(x)\geqslant r$ for $i\notin I$. In this case, sometimes we
say that every finite subset of the sequence $f_n$ is shattered by
$X$. If $F$ has not the independence property then we say that it
has the {\em dependent property} (or the NIP).
\end{dfn}

We have the following remarkable topological characterizations of
this property. More details and several equivalent presentations
can be found in \cite{GM}.

\begin{fct}[\cite{GM}, Theorem~2.11] \label{NIP-convergence}
 Let $X$ be a compact space and $F\subseteq C(X)$ a bounded subset.
The following conditions are equivalent:
\begin{itemize}
             \item [{\em (i)}] $F$ does not contain an independent  sequence.
             \item [{\em (ii)}] Each sequence in $F$ has a pointwise convergent subsequence in $\mathbb{R}^X$.
\end{itemize}
\end{fct}


\begin{dfn} \label{RSC}
 We say that a (bounded) family  $F$ of real-valued function on
a set $X$ has the {\em relative sequential compactness in
${\mathbb{R}}^X$} (short RSC) if   every sequence in $F$ has a
pointwise convergent subsequence in ${\mathbb{R}}^X$.
\end{dfn}

As we will see shortly, the following statement is a
generalization of a model theoretic fact, i.e. IP implies OP.

\begin{fct} \label{IP->OP} Let $X$ be a compact space and $F\subseteq C(X)$ a bounded subset.
 If $F$ is relatively weakly compact in $C(X)$, then $F$ has the RSC.
\end{fct}
\begin{proof}
Suppose that $F$ is relatively weakly compact in $C(X)$.  (Not
that, by Fact~\ref{Grothendieck-lemma} above, the weak topology
and pointwise topology are the same.) By the Eberlein-\v{S}mulian
theorem, each sequence in $F$ has a subsequence converging to an
element of  $C(X)$. So, in particular, $F$ has the RSC.
\end{proof}

\subsection{Talagrand's stability and almost NIP}
Historically, Talagrand's stability (see
Definition~\ref{Talagrand-stable} below), which we call the almost
dependence property, arose naturally when Talagrand and Fremlin
were studying pointwise compact sets of measurable functions;
they found that in many cases a set of functions was relatively
pointwise compact because it was almost dependent (see
Fact~\ref{almost-NIP} below). Later it appeared  that the concept
was connected with Glivenko-Cantelli classes in the theory of
empirical measures, as explained in \cite{Talagrand}. In this
subsection we study this property and show that it is the
`correct' counterpart of NIP in integral logic (see \cite{K}).
Then, we point out the connection between NIP in continuous logic
and this property.

\begin{dfn}[Talagrand's stability, \cite{Fremlin4}, 465B]  \label{Talagrand-stable}
  Let $A\subseteq C(X)$ be a pointwise bounded family
of real-valued continuous functions  on $X$. Suppose that $\mu$
is a measure on $X$. We say that $A$ is {\em $\mu$-stable}, if
$A$ is a stable set of functions in the sense of Definition~465B
in \cite{Fremlin4}, that is, whenever $E\subseteq M$ is
measurable, $\mu(E)>0$ and $s<r$ in $\mathbb{R}$, there is some
$k\geqslant 1$ such that $(\mu^{2k})^*D_k(A, E,s,r)<(\mu E)^{2k}$
where
\begin{align*}
D_k(A, E,s,r) = \bigcup_{f\in A}\big\{w\in &
E^{2k}:f(w_{2i})\leqslant s, ~f(w_{2i+1})\geqslant r  \textrm{ for
} i<k\big\}.
\end{align*}
\end{dfn}

Now we invoke the first result connecting this notion. First, we
need a notion and a notation. If $X$ is any set and $A$ a subset
of ${\Bbb R}^X$, then the topology of {\em pointwise convergence}
on $A$ is that inherited from the usual product topology of ${\Bbb
R}^X$; that is, the coarsest topology on $A$ for which the map
$f\to f(x) : A\mapsto {\Bbb R}$  is continuous for every $x\in X$.
 We will denote the pointwise closure of $A$ in ${\Bbb R}^X$ by $cl_p(A)$.

\begin{fct}[{\cite[465D]{Fremlin4}}]
\label{almost-NIP} Let $X$ be a compact Housdorff space and
$A\subseteq C(X)$ be a pointwise bounded family of real-valued
continuous functions from $X$. Suppose that $\mu$ is a Radon
measure on $X$. If $A$ is $\mu$-stable, then $cl_p(A)$ is
$\mu$-stable and every element in $cl_p(A)$ is $\mu$-measurable.
\end{fct}

In \cite{Fremlin75} Fremlin obtained a remarkable result, which
has become known as Fremlin's dichotomy:  a set of measurable
functions on a perfect measure space is either `good' (relatively
countably compact for the pointwise topology and relatively
compact for the topology of convergence in measure) or `bad'
(with neither property). We recall that a subset $A$ of a
topological space $X$ is {\em relatively countably compact} if
every sequence of $A$ has a cluster point in $X$.

\begin{fct}[Fremlin's dichotomy, \cite{Fremlin4}, 463J]
Let $(X,\Sigma,\mu)$ be a perfect $\sigma$-finite measure space,
and $\{f_n\}$ a sequence of real-valued measurable functions on
$X$. Then
\begin{enumerate}
\item[]   either $\{f_n\}$ has a subsequence which is convergent almost everywhere
\item[]   or $\{f_n\}$ has a subsequence with no measurable cluster point in $\mathbb{R}^X$.
\end{enumerate}
\end{fct}

\medskip
We now define the notion of $\mu$-almost NIP and we will see
shortly the connection between this notion and NIP.

\begin{dfn}[$\mu$-almost NIP]  Let $A\subseteq C(X)$ be a pointwise bounded family
of real-valued continuous functions  on $X$. Suppose that $\mu$
is a measure on $X$. We say that $A$ has the {\em $\mu$-almost
NIP}, if every sequence in $A$ has a subsequence which is
convergent $\mu$-almost everywhere.
\end{dfn}

Let $(X,\Sigma,\mu)$ be a finite Radon measure on a compact space
$X$, and  $\mathcal{L}^0$ the set of all real-valued measurable
functions on $X$.  Let $A\subseteq \mathcal{L}^0$ be a  bounded
family. Then we say that $A$ satisfies condition (M), if for all
$s<r$ and all $k$, the set $D_k(A,X,r,s)$ is  measurable (this
applies, in particular, if $A$ is countable).

\begin{proposition} \label{NIP-Fremlin} Let $(X,\Sigma,\mu)$ be a finite Radon measure on a compact space $X$, and
$A\subseteq \mathcal{L}^0$ a  bounded family of real-valued
measurable functions on $X$. Consider the following statements.
\begin{itemize}
             \item [{\em (i)}] $A$ is $\mu$-stable.
             \item [{\em (ii)}] There do  not exist measurable set $E$  with  $\mu(E)>0$ and $s<r$ in $\mathbb{R}$, such that for each $n$, and almost all $w\in E^n$,
             for each subset $I$ of $\{1,\ldots,n\}$, there is $f\in A$ with $$f(w_i)<s \text{ if } i\in I \text{ and } f(w_i)>r \text{ if } i\notin I.$$
             \item [{\em (iii)}]  $A$ has  the $\mu$-almost NIP.
\end{itemize}
Then (i)~$\Rightarrow$~(ii). If $A$ satisfies condition (M), then
(ii)~$\Rightarrow$~(i). (i)~$\Rightarrow$~(iii), but
(iii)~$\nRightarrow$~(i) and (iii)~$\nRightarrow$~(ii).
\end{proposition}
\begin{proof} (i)~$\Rightarrow$~(ii) is evident.

(M)$\wedge$(ii)~$\Rightarrow$~(i) is Proposition~4 in
\cite{Talagrand}.

 (i)~$\Rightarrow$~(iii):   Let $\{f_n\}$ be any sequence in $A$, and take an arbitrary
subsequence of it (still denoted by $\{f_n\}$). Let $\mathcal{D}$
be a non-principal ultrafilter on $\mathbb{N}$, and then define
$f(x)=\lim_{\mathcal{D}} f_i(x)$ for all $x\in X$. (By the
assumption, there is a real number $r$ such that $|h|\leqslant r$
for each $h\in A$, and therefore $f$ is well defined.)  Since $A$
 is $\mu$-stable and $f\in \text{cl}_p(\{f_n\})$, the
function $f$ is measurable (see Fact~\ref{almost-NIP}). So every
subsequence of $\{f_n\}$ has a measurable cluster point.
Fremlin's dichotomy now tells  us that $\{f_n\}$ has a subsequence
which is convergent almost everywhere.

(iii)~$\nRightarrow$~(i)$\vee$(ii): In \cite{SF} Shelah and
Fremlin found that in a model of set theory there is a separable
pointwise compact set $A$ of real-valued Lebesgue measurable
functions on the unit interval which it is not $\mu$-stable. Thus
we see that (iii)~$\nRightarrow$~(i). Since the set $A$ is
separable, it satisfies condition (M) and therefore (ii) fails.
\end{proof}

Professor Fremlin kindly pointed out to us that Shelah's model,
described in their paper \cite{SF}, in fact deals with the point
that there is a countable set of \emph{continuous} functions which
is relatively pointwise compact in ${\mathcal L}^0(\mu)$ for a
Radon measure $\mu$, but that it is not $\mu$-stable. Of course,
in some cases, there are still some things to say (see
Theorem~\ref{NIP-compactness} below).

For a Hausdorff space $X$, $\textbf{M}_r(X)$ will be the space of
universally measurable functions, i.e. a function $f$ is an
element of $\textbf{M}_r(X)$ iff $f$ is $\mu$-measurable for
every Radon measure $\mu$ on $X$.

\begin{fct}[BFT Criterion, \cite{BFT}, Theorem~2F] \label{BFT}
 Let $X$ be a compact Hausdorff space, and $F\subseteq C(X)$ be
bounded. Then the following are equivalent.
\begin{itemize}
             \item [{\em (i)}] $F$ has the NIP (see Definition~\ref{NIP-family} above).
             \item [{\em (ii)}] $F$ is relatively compact in $\textbf{M}_r(X)$ for the topology of pointwise convergence.
             \item [{\em (iii)}] $F$ has the RSC (see Definition~\ref{RSC} above).
             \item [{\em (iv)}] Each sequence in $F$ has a subsequence which is convergent $\mu$-almost everywhere for every Radon measure $\mu$ on $X$.
             \item [{\em (v)}] For each Radon measure $\mu$ on $X$, each sequence in $F$ has a subsequence which is convergent $\mu$-almost everywhere.
\end{itemize}
\end{fct}
\begin{proof}  The equivalence (i)--(iii) is the equivalence
(ii)~$\Leftrightarrow$~(vi)~$\Leftrightarrow$~(iv) of Theorem~2F
of \cite{BFT}. (See also Fact~\ref{NIP-convergence} above.)

 Fremlin's dichotomy and the equivalence
(v)~$\Leftrightarrow$~(vi)~$\Leftrightarrow$~(iv) of Theorem~2F
of \cite{BFT} imply
(v)~$\Leftrightarrow$~(i)~$\Leftrightarrow$~(iv).
\end{proof}

We will see that the BFT criterion  in NIP theories plays a role
similar to the role played by the Eberlein-Grothendieck criterion
in stable theories.

\subsection{NIP in a model}
In \cite{Sh} Shelah introduced the independence property (IP) for
0-1 valued formulas; a formula $\phi(x,y)$ has the IP if for each
$n$ there exist $b_1,\ldots,b_n$ in the monster model such that
each nontrivial Boolean combination of
$\phi(x,b_1),\ldots,\phi(x,b_n)$ is satisfiable. By some
set-theoretic considerations, a formula $\phi(x,y)$ has IP if and
only if $\sup\{|S_\phi(A)|:A\text{ of size }\kappa\}=2^\kappa$
for some infinite cardinal $\kappa$. Although this property was
introduced for counting types, its negation (NIP) is a successful
extension of local stability and also an active domain of
research in classical first order logic and other areas of
mathematics. The following generalization of NIP (in the
framework of continuous logic) also has a natural topological
presentation.

\begin{dfn}  \label{NIP-formula}
 Let $M$ be a structure, and $\phi(x,y)$ a formula.
We say that $\phi(x,y)$ is NIP on $M\times M$ (or on $M$) if for
each sequence $(a_n)\subseteq M$, and
                       $r>s$, there are some {\em finite} disjoint subsets $E,F$ of ${\Bbb N}$ such that $$\Big\{b\in M:\big( \bigwedge_{n\in E}\phi^{\M}(a_n,b)\leqslant
                        s\big)\wedge\big(\bigwedge_{n\in F}\phi^{\M}(a_n,b)\geqslant r\big)\Big\}=\emptyset.$$
\end{dfn}

\begin{lem} \label{equivalence}
 Let $M$ be a structure, and $\phi(x,y)$ a formula.
Then the following are equivalent.
\begin{itemize}
            \item [(i)]  $\phi(x,y)$ is NIP on $M$.
            \item [(ii)] For each sequence $(a_n)\subseteq M$, each saturated elementary extension ${N}\succeq{M}$, and
                       $r>s$, there are  disjoint subsets $E,F$ of ${\Bbb N}$ such that $$\Big\{b\in N: \big( \bigwedge_{n\in E}\phi(a_n,b)\leqslant
                        s\big)\wedge\big(\bigwedge_{n\in F}\phi(a_n,b)\geqslant r\big)\Big\}=\emptyset.$$
           \item [(iii)] For each sequence $\phi(a_n,y)$ in the set
                        $A=\{\phi(a,y):S_y(M)\to{\Bbb R}~|~a\in M\}$, where $S_y(M)$ is
                        the space of all complete types on $M$ in the variable $y$, and $r>s$
                        there are  {\em finite} disjoint subsets $E,F$ of ${\Bbb N}$ such that  $$\Big\{y\in S_y(M):\big( \bigwedge_{n\in E}\phi(a_n,y)\leqslant
                        s\big)\wedge\big(\bigwedge_{n\in F}\phi(a_n,y)\geqslant r\big)\Big\}=\emptyset.$$
           \item [(iv)] The condition (iii) holds for
                          {\em arbitrary} disjoint subsets $E,F$ of ${\Bbb N}$.
           \item [(v)] Every sequence $\phi(a_n,y)$ in $A$ has a convergent subsequence in ${\Bbb R}^X$, equivalently $A$ has the RSC.
\end{itemize}
\end{lem}
\begin{proof}
(ii)~$\Rightarrow$~(i) follows from the compactness theorem and
saturation. (Indeed, suppose that (i) fails, and then consider a
suitable type and get a contradiction.)

(i)~$\Rightarrow$~(iii) is just a restatement of the notion of
type.

 (iv)~$\Rightarrow$~(iii): Suppose that (iii) fails for the sequence $(a_n)\subseteq M$ and $s<r$.  Let  $E=\{i_n:n\in \Bbb
N\}$ and $F=\{j_n:n\in \Bbb N\}$ be arbitrary disjoint subsets of
${\Bbb N}$. Let $E_m=\{i_n:n\leq m\}$
 and $F_m=\{j_n:n\leq m\}$ for each $m\in\Bbb N$.
So, for each $m$, there is some $y_m\in S_y(M)$ such that $
\bigwedge_{n\in E_m}\phi(a_n,y_m)\leqslant s$ and
$\bigwedge_{n\in F_m}\phi(a_n,y_m)\geqslant r$. Let $z$ be a
cluster point of the sequence $(y_m)$. (Note that $z\in S_y(M)$
because the type space is compact.) Since the $\phi(a_n,y)$  are
continuous, it is easy to verify that $ \bigwedge_{n\in
E}\phi(a_n,z)\leqslant s$ and $\bigwedge_{n\in
F}\phi(a_n,z)\geqslant r$. As $E,F$ are arbitrary, (iv) fails.
(iii)~$\Rightarrow$~(iv) is evident.

(iii)~$\Leftrightarrow$~(v):  It is easy to verify that for the
set $A$, the dependence property in Definition~\ref{NIP-family}
is equivalent to the condition (iii). Now,  by
Fact~\ref{NIP-convergence} the proof is completed.

(ii)~$\Leftrightarrow$~(iv) follows from saturation (and the
notion of type).
\end{proof}

 Some similar notions are studied in  \cite{Iba14}.  Note that
the notion NIP on a model is `double local', i.e. $\phi$ can be
NIP on a model, but not in a theory.

\medskip
If $\phi(x,y)$ is a formula, we let $\tilde{\phi}(y, x)=
\phi(x,y)$. Hence $\tilde{\phi}$ is the same formula as $\phi$,
but we have exchanged the role of variables and parameters.

\begin{rmk}  Let $M$ be a structure and $\phi(x,y)$ a formula.
The space $S_{\tilde\phi}(M)$ of all $\tilde\phi$-types on $M$ is
the quotient of $S_y(M)$ given by the family of functions
$\{\phi(a,y):a\in M\}$ (see \cite{BU}, Fact~4.7). So in
Definition~\ref{NIP-formula} above, $S_y(M)$ can be replace by
$S_{\tilde\phi}(M)$.
\end{rmk}

\begin{thm}[NIP and $\mu$-stability] \label{NIP-compactness}
Let $M$ be an $\aleph_0$-saturated $L$-structure, $\phi(x;y)$ a
formula, $A=\{\phi(a,y):a\in M\}$ and $\tilde{A}=\{\phi(x,b):b\in
M\}$. Then the following are equivalent:
\begin{itemize}
             \item [{\em (i)}]  $\phi$ is NIP on $M$.
             \item [{\em (ii)}] $\tilde A$ is $\mu$-stable for all Radon measures $\mu$ on $S_\phi(M)$.
\end{itemize}
\end{thm}
\begin{proof} (i)~$\Rightarrow$~(ii): By the compactness theorem of
continuous logic, since $M$ is $\aleph_0$-saturated and
$\phi(x,y)$ is NIP on $M$, there is some integer $n$ such that no
subset (of $M$) of size $n$ is shattered by $\phi(x,y)$. We note
that by Proposition~465T of \cite{Fremlin4}, the conditions (i)
and (ii) of Proposition~\ref{NIP-Fremlin} are equivalent. So if
$E\subseteq M$, $\mu(E)>0$, $r>s$, then for each
$(a_1,\ldots,a_n)\in E^n$ there is a set $I\subseteq
\{1,\ldots,n\}$ such that
$$\Big\{y\in S_y(M):\big( \bigwedge_{i\in I}\phi(a_i,y)\leqslant
s\big)\wedge\big(\bigwedge_{i\notin I}\phi(a_i,y)\geqslant
r\big)\Big\}=\emptyset,$$ where $S_y(M)$ is the space of all
complete types on $M$ in the variable $y$. Since $M\subseteq
S_y(M)$, the set ${\tilde A}$ is $\mu$-stable for every Radon
measure $\mu$ on $S_\phi(M)$.

(ii)~$\Rightarrow$~(i): Suppose that $\tilde A$ is $\mu$-stable
for every Radon measure $\mu$ on $\tilde{X}=S_\phi(M)$.  Thus, by
Fact~\ref{almost-NIP}, $\tilde A$ is relatively compact in
$\textbf{M}_r({\tilde X})$ (the space of all $\mu$-measurable
functions on $\tilde X$ for each Radon
 measure $\mu$ on ${\tilde X}$). By the BFT criterion,
for each sequence $\phi(x,a_n)$ in $\tilde A$, and  $r<s$, there
is some $I\subseteq {\mathbb{N}}$ such that $$\Big\{x\in
S_x(M):\big( \bigwedge_{i\in I}{\phi}(x,a_i)\leqslant
s\big)\wedge\big(\bigwedge_{i\notin I}{\phi}(x,a_i)\geqslant
r\big)\Big\}=\emptyset.$$ Thus the dual formula ${\tilde
\phi}(y,x)$ is NIP on $M$. So, by applying the direction
(i)~$\Rightarrow$~(ii) to the formula $\tilde \phi$, we see that
${\tilde{\tilde A}}=A$ is $\mu$-stable for every Radon measure
$\mu$ on   $X=S_{\tilde{\phi}}(M)$. Thus, again by the BFT
criterion and Proposition~\ref{NIP-Fremlin}, we conclude that
$\phi(x,y)$ is NIP on $M$.
\end{proof}

In fact the proof of the previous result says more: if $M$  is
$\aleph_0$-saturated, then $\phi$ is NIP on $M$ if and only if
$\tilde \phi$ is NIP on $M$.

\begin{cor}  Under the assumptions in Theorem~\ref{NIP-compactness}, $\phi$ is NIP on $M$ if and only if $A$
is $\mu$-stable for every Radon measure $\mu$ on
$S_{\tilde\phi}(M)$.
\end{cor}

The previous results also show  why  the $\mu$-stability is the
`correct' notion of NIP in integral logic (see \cite{K}).

\medskip
The following is a translation of the BFT criterion into
(continuous) model theory. Note that here we do not need any
saturation conditions on the model.

\begin{proposition}[NIP and $\mu$-almost NIP]  \label{NIP-almost}
Let $M$ be an  $L$-structure, $\phi(x;y)$ a formula and
$A=\{\phi(a,y):a\in M\}$. Then the following are equivalent:
\begin{itemize}
             \item [{\em (i)}]  $\phi$ is NIP on $M$.
             \item [{\em (ii)}] $A$ has the $\mu$-almost NIP for all Radon measures $\mu$ on $S_{\tilde\phi}(M)$.
 \end{itemize}
\end{proposition}
\begin{proof}
This is the equivalence (i)~$\Leftrightarrow$~(v) of
Fact~\ref{BFT} with $X=S_{\tilde\phi}(M)$ and $F=A$.
\end{proof}

\begin{rmk} One can not expect  the notion `NIP on a model'
to be symmetrical. It is easy to make examples such that $\phi$ is
NIP on $M$ but $\tilde\phi$ is not NIP on $M$. Of course, if $M$
is $\aleph_0$-saturated, they are the same.
\end{rmk}

\subsection{Almost definable coheirs} \label{3}
It is well known that every type on a stable model is definable
(see \cite{Ben-Gro}). Here we want to give a counterpart of this
fact for NIP theories. In \cite{K} it is shown that if a formula
$\phi$ (in integral logic) is $\mu$-stable on a model $M$, then
every type in $S_\phi(M)$ is $\mu$-almost definable.

We present the notion of  `coheir' here. Let $M^*$  be a
saturated elementary extension of $M$. A type $p(x)\in
S_\phi(M^*)$ is called a {\em coheir (of a type) over $M$} if for
every condition $\varphi=0$ in $p(x)$ and every $\epsilon>0$, the
condition $|\varphi|\leq \epsilon$ is satisfiable in $M$. (In
classical ($\{0,1\}$-valued) model theory, this means that every
formula in $p(x)$ is realized in $M$.)  In this case we say that
$p$ is  $M$-finitely satisfiable. It is easy to verify that a
type $p(x)\in S_\phi(M^*)$ is a coheir over $M$ if and only if
there are $(a_i\in M: i\in I)$ and an ultrafilter $\mathcal D$ on
$I$ such that $\lim_{i,\mathcal D}tp(a_i/M)=p$, where the
$\mathcal D$-limit is taken in the logic topology. (For the
definition of $\mathcal D$-limit, see Section~5 of \cite{BBHU}.)
By Proposition~\ref{key}, $p$ is a coheir over $M$ if and only if
there are $(a_i\in M: i\in I)$ and an ultrafilter $\mathcal D$ on
$I$ such that $\lim_{i,\mathcal D}I_{p_i}=I_p$, where
$p_i=tp(a_i/M)$ and the $\mathcal D$-limit is taken in the weak*
topology on $\sigma_{M^*}(M^*)$.

\medskip
Here we say a function $\psi:X\to \mathbb{R}$ on a topological
space $X$ is {\em universally measurable}, if it is
$\mu$-measurable for every probability Radon measure $\mu$ on $X$.

\begin{dfn} \label{universal dfn}  Let $M^*$ be a saturated elementary extension
of $M$ and $p(x)\in S_\phi(M^*)$ be a coheir of a type over $M$.
 We say that a universally measurable function
$\psi:S_{\tilde\phi}(M)\to\mathbb{R}$ {\em defines} $p$  if
$\phi(p,a)=\psi(tp_{\tilde\phi}(a/M))$ for all $a\in M^*$, and in
this case we say that $p$ is {\em universally definable}.
\end{dfn}

The above notion is well defined since every coheir is
$M$-finitely satisfiable and so $M$-invariant.

\begin{rmk}[\cite{Pillay-Grothendieck}, Remark~2.1] There is a
correspondence between the set of all coheirs of types over $M$
and the closure of the set
$A=\{\phi^a(y):S_{\tilde\phi}(M)\to{\mathbb{R}}~|a\in M\}$, where
$\phi^a(q)=\phi(a,q)$ for all $q\in S_{\tilde\phi}(M)$. Indeed,
let $M^*$ be a saturated elemetary extension of $M$.
 Then for any global $M$-finitely satisfiable $\phi$-type $p(x)\in S_\phi(M^*)$ there
is a function $\psi_p$ in the closure $A$ such that
$\psi_p(tp_{\tilde\phi}(a/M))=\phi(p,a)$ for all $a\in M^*$.
Indeed, suppose that $tp_\phi(a_i/M^*)\to p$ in the logic
topology, where $a_i\in M$. Define $\psi_p(y)=\lim_i\phi(a_i,y)$
for all $y\in S_{\tilde\phi}(M)$. Now, it is easy to verify that
$\psi_p(tp_{\tilde\phi}(a/M))=\phi(p,a)$ for all $a\in M^*$.  To
summarize, let $S_\phi^{M\text{-fs}}(M^*)$ be the set of all
global $M$-finitely satisfiable $\phi$-types (over the monster
model $M^*$). Then the map $S_\phi^{M\text{-fs}}(M^*)\to
\overline{A}$, defined by $p\mapsto\psi_p$, is a homeomorphic
embedding of $S_\phi^{M\text{-fs}}(M^*)$ in the pointwise
convergence topology on $\overline{A}\subseteq {\Bbb
R}^{S_{\tilde\phi}(M)}$.
\end{rmk}

For simplicity, we will write $\phi(p,a)=\phi(p,b)$ where
$b=tp_{\tilde\phi}(a/M)$ and $a\in M^*$.

\medskip
The following is  a translation of the BFT criterion:

\begin{proposition} \label{NIP-dfn} Let $M$ be a
structure and $\phi(x,y)$ a formula. Then the following are
equivalent:
\begin{itemize}
             \item [{\em (i)}] $\phi$ is NIP on $M$.
             \item [{\em (ii)}] Every  coheir of a $\phi$-type over $M$  is definable by a
             universally measurable relation $\psi(y)$ over $S_{\tilde\phi}(M)$.
\end{itemize}
\end{proposition}
\begin{proof} (i)~$\Rightarrow$~(ii):  Let
$A=\{\phi^a(y):S_{\tilde\phi}(M)\to{\mathbb{R}}~|a\in M\}$. By
NIP, $A$ is relatively compact in
$\textbf{M}_r(S_{\tilde\phi}(M))$ (see the BFT criterion).
Suppose that $p_{a_i}\to p\in S_{\phi}(M^*)$
 where $p_{a_i}$ is realized by $a_i\in M$ and $M^*$ is a saturated elementary extension of $M$.
 (We note that the set of all types realized in $M$ is dense in
the set of all coheirs.) Thus $\phi^{a_i}\to\psi$ pointwise where
$\psi$ is universally measurable, and $\psi$ defines $p$.

(ii)~$\Rightarrow$~(i): Suppose that  $\phi^{a_i}\to\psi$
pointwise. We can assume that $p_{a_i}\to p\in S_{\phi}(M^*)$.
Suppose that $p$ is definable by a universally measurable relation
$\varphi$, so we have $\psi=\varphi$ on $S_{\tilde\phi}(M)$.  So,
$\psi$ is measurable for all Radon measures on
$S_{\tilde\phi}(M)$. Again by the BFT criterion, $\phi$ is NIP.
\end{proof}

\begin{rmk}  \label{invariant-type}
In \cite{HP}, the authors  showed that, in 0-1 valued logic,
every global invariant type admits a Borel definition assuming
NIP. This implies, in particular, that every global $M$-invariant
type admits a universally measurable definition. Note that
Proposition~\ref{NIP-dfn} is an extension  of their result to
continuous logic for global finitely satisfiable types.
 Moreover, one can
show that  $\phi(x,y)$ is NIP on a  separable model $M$ if and
only if  every global $M$-finitely satisfiable $\phi$-type is
Baire 1 definable (see also Fact~\ref{Polish-compact},
Theorem~\ref{Baire-dfn} and Remark~\ref{NIP=Baire 1} below).
\end{rmk}

Here we mention a characterization of NIP in
 terms  of measure algebra. For this, a definition is needed. Let
$\phi(x,y)$ be a formula, $r\in\mathbb{R}$ and $a\in M$. By
$\{\phi(x,a)\geqslant r\}$ we denote the set $\{p\in
S_\phi(M):\phi(p,a)\geqslant r \}$. The set $\{ \phi(x,a)\leqslant
r\}$ has the obvious meaning. The measure algebra generated by
$\phi$ on $S_\phi(M)$ is the measure algebra generated by all
sets of the forms $\{\phi(x,a)\geqslant r\}$ and
$\{\phi(x,b)\leqslant s\}$ where $a,b\in M$ and $r,s\in
\mathbb{R}$. One can assume that all $r,s$ are rational numbers.
Now, a straightforward translation of the proof for classical
first order theories, as can be found in
\cite[Theorem~3.14]{Keisler}, implies that:

\begin{proposition} \label{Keisler-NIP} Let $T$ be a theory and $\phi(x,y)$ a formula.
Then the following are equivalent:
\begin{itemize}
             \item [{\em (i)}] $\phi$ is NIP.
             \item [{\em (ii)}] For every sufficiently saturated model $M$, each Radon measure on $S_\phi(M)$ has a
             countably generated measure algebra (which is the
             measure algebra generated by $\phi$).
\end{itemize}
\end{proposition}

Now we are going to give another characterization of NIP. First we
need some definitions. Let $\psi$ be a measurable function on
$(S_{\tilde\phi}(M),\mu)$ where $\mu$ is a probability Radon
measure on $S_{\tilde\phi}(M)$. Then $\psi$ is called an {\em
almost ${\tilde\phi}$-definable relation over $M$} if there is a
sequence $g_n:S_{\tilde\phi}(M)\to \mathbb{R}$, $|g_n|\leqslant
|{\tilde\phi}|$, of continuous functions such that $\lim_n
g_n(p)=\psi(p)$ for almost all $p\in S_{\tilde\phi}(M)$. (We note
that by the Stone-Weierstrass theorem every continuous function
$g_n:S_{\tilde\phi}(M)\to \mathbb{R}$ can be expressed as a
uniform limit of algebraic combinations of (at most countably
many) functions of the form $p\mapsto {\tilde\phi}(p, b)$, $b\in
M$.)

 An almost
$\tilde\phi$-definable relation $\psi(y)$ over $M$ defines a
coheir $p(x) \in S_\phi(M^*)$ (of a $\phi$-type over $M$) if the
set $A_0\subseteq S_{\tilde\phi}(M)$ is measurable and
$\mu(A_0)=1$, where $A_0=\{b\in S_{\tilde\phi}(M):\phi(p,b)
=\psi(b)\}$. In this case we say that $p$ is ($\mu$-)\emph{almost
definable}. It is easy to check that almost definability is well
defined.  Suppose that every coheir $p$  is almost definable by a
measurable function $\psi^p$.
 Then, we say that $p$ is {\em almost equal to $q$},
denoted by $p\equiv_\mu q$, if $\psi^p=\psi^q$  $\mu$-almost
everywhere. For a coheir $p(x)$, define  $[p]_\mu=\{q\in
S_\phi(M^*): p\equiv_\mu q  \text{ and $q$ is a coheir}\}$ and
$[S_\phi]_\mu(M)=\{[p]_\mu:p\in S_\phi(M^*) \text{ is a
coheir}\}$. Then $[S_\phi]_\mu(M)$ has a natural topology which
is defined by metric $d([p]_\mu,[q]_\mu)=\int|\psi^p-\psi^q|d\mu$
for coheirs $p,q\in S_\phi(M^*)$.

\medskip
Recall that the density character of a topological space $X$, is
the least infinite cardinal number of a dense subset of $X$. When
measuring the size of a structure we will use its density
character (as a metric space), denoted $\|M\|$, rather than its
cardinality. Similarly, since $[S_\phi]_\mu(M)$ is a metric
space, we measure the size $[S_\phi]_\mu(M)$ by its density
character $\|[S_\phi]_\mu(M)\|$.

\begin{thm}[Almost definability of coheirs] \label{almost-dfn}
 Let $T$ be a theory and $\phi(x,y)$ a formula. Then the following are equivalent:
\begin{itemize}
             \item [{\em (i)}] $\phi$ is NIP.
              \item [{\em (ii)}] For every model $M$ and measure $\mu$ on $S_{\tilde\phi}(M)$, every
              coheir of a type over $M$ is $\mu$-almost definable, and $\|[S_{\phi}]_\mu(M)\|\leqslant\|M\|$.
\end{itemize}
\end{thm}
\begin{proof} (i)~$\Rightarrow$~(ii):  Suppose that  the coheir $p\in S_\phi(M^*)$ is
definable by a universally measurable relation $\psi$ on
$S_{\tilde\phi}(M)$. Let $\mu$ be a Radon measure on
$S_{\tilde\phi}(M)$. Then there is a sequence $g_n$ of continuous
functions on $S_{\tilde\phi}(M)$ such that $g_n\to \psi$ in
$L^1(\mu)$ (see \cite[7.9]{Folland}), and hence
a subsequence (still denoted by $g_n$) that converges to $\psi$
$\mu$-almost everywhere. So $p$ is $\mu$-almost definable.
Moreover, by the Stone-Weierstrass theorem,
$\|C(S_{\tilde\phi}(M))\|\leqslant\|M\|$. Now, since
$C(S_{\tilde\phi}(M))$ is dense in $L^1(\mu)$ (again see
\cite[7.9]{Folland}), $\|L^1(\mu)\|\leqslant\|M\|$. By definition,
$\|[S_{\phi}]_\mu(M)\|\leqslant\|M\|$ and the proof is completed.

(ii)~$\Rightarrow$~(i): Let $p\in S_\phi(M^*)$ be a coheir of a
type over $M$. Suppose that $p_{a_i}\to p$ where $p_{a_i}$ is
realized by $a_i\in M$. Then the function
$\psi(y)=\lim_i\phi(a_i,y)$ is measurable for all Radon measures
on $S_{\tilde\phi}(M)$. Indeed, by definition, for each Radon
measure $\mu$, there is a measurable function $\psi_\mu$ such
that $\psi_\mu(b)=\phi(p,b)$ $\mu$-almost everywhere. Since $\mu$
is Radon (and so is complete), and $\psi=\psi_\mu$ almost
everywhere, $\psi$ is $\mu$-measurable (see
\cite[2.11]{Folland}). Then, by Proposition~\ref{NIP-dfn}, the
proof is completed.
\end{proof}

\subsection{Baire 1 definable types}
More results can be reached, if one works in a separable model.
Let $X$ be a Polish space. A function $f:X\to{\mathbb{R}}$ is of
Baire class 1 if it can be written as the pointwise limit of a
sequence of continuous functions. The set of Baire class 1
functions on $X$ is denoted by $B_1(X)$.

\begin{fct}[BFT Criterion for Polish spaces, \cite{BFT}, Corollary~4G]  \label{Polish-compact}
Let $X$ be a Polish space, and $A\subseteq C(X)$ pointwise
bounded set. Then the following are equivalent:
\begin{itemize}
             \item [{\em (i)}] $A$ is relatively compact in $B_1(X)$.
              \item [{\em (ii)}] $A$ is relatively sequentially compact in ${\mathbb{R}}^X$, or $A$ has the RSC.
\end{itemize}
\end{fct}

Fremlin's notion of an angelic topological space is as follows: a
regular Hausdorff space X is {\em angelic} if (i) every
relatively countably compact set in $X$  is relatively compact,
(ii) the closure of a relatively compact set is precisely the set
of limits of its sequences. The following is the principal result
of \cite{BFT}.

\begin{fct}[\cite{BFT}, Theorem~3F] \label{Polish-angelic}
If $X$ is a Polish space, then $B_1(X)$ is angelic under the
topology of pointwise convergence.
\end{fct}

Let $M$ be a structure and $\phi(x,y)$ a formula. A Baire class 1
function $\psi:S_{\tilde\phi}(M)\to{\mathbb{R}}$ defines $p\in
S_\phi(M)$ if $\phi(p,b)=\psi(b)$ for all $b\in M$. We say $p$ is
Baire 1 definable if some Baire class 1 function $\psi$ defines
it. The following is another criterion for NIP.

\begin{thm}[Baire 1 definability of types] \label{Baire-dfn}
 Let $\phi(x,y)$ be a NIP formula
 on a separable model $M$. Then every $p\in S_{\phi}(M)$ is
definable by a Baire 1 function $\psi(y)$ on $S_{\tilde\phi}(M)$.
\end{thm}
\begin{proof} The proof is an easy consequence of
Fact~\ref{Polish-compact}. Suppose that $p_{a_i}\to p\in
S_{\phi}(M)$ where $a_i\in M$. (Recall that the set of all types
realized in $M$ is dense in $S_{\phi}(M)$.) For each $a\in M$,
define $\phi^a:M\to {\Bbb R}$ by $\phi^a(b)=\phi(a,b)$.
 Since $\phi$ is
NIP, the set $\hat{A}=\{\phi(a,y):S_y(M)\to{ \mathbb{R}}:a\in
M\}$ is relatively sequentially compact in ${\Bbb R}^{S_y(M)}$,
and in  particular  the set $A=\{\phi^a:a\in M\}$ is relatively
sequentially compact in ${\Bbb R}^M$. Now by
Fact~\ref{Polish-compact}, since $M$ is Polish,  also $A$ is
relatively compact in $B_1(M)$. Thus, there is a $\psi\in B_1(M)$
such that $\phi^{a_i}\to \psi$, so $p$ is definable by a Baire
class 1 function. Moreover, since $B_1(M)$ is angelic, there is
some sequence $\phi^{a_n},a_n\in M$ such that $\phi^{a_n}\to\psi$.
\end{proof}

\begin{rmk} \label{NIP=Baire 1}  Note that one can say
more: $\phi$ is NIP on $M$ if and only if every coheir is Baire 1
definable. This is discussed in detail in \cite{KP}.
\end{rmk}



\noindent\hrulefill

\section{SOP} \label{5}
 In this section we work in the classical
logic. One reason for restricting our attention to the classical
case is to make this section  more accessible to model-theorists
and other interested readers.

 In \cite{Sh} Shelah introduced the
strict order property as complementary to the independence
property: a theory has OP iff it has IP or SOP. In functional
analysis, the Eberlein-\v{S}mulian theorem states that a subset
of a Banach space is not relatively weakly compact iff it has a
sequence without any weak Cauchy subsequence or it has a weak
Cauchy sequence with no weak limit. In fact there is a
correspondence  between the Eberlein-\v{S}mulian theorem and
Shelah's result above. To determine this correspondence, we first
give a topological description of the strict order property, and
then study the above dividing line.

 In classical ($\{0,1\}$-valued) model
theory a formula $\phi(x,y)$ has the {\em strict order property}
(or short SOP) if there exists a sequence $(a_i:i<\omega)$ in the
monster model $\mathcal{U}$ such that for all $i<\omega$,
$$\phi({\mathcal{U}},a_i)\subsetneqq\phi({\mathcal U},a_{i+1}).$$
The acronym SOP stands for the strict order property and NSOP is
its negation. We can assume that $\phi(x,y)$ is a 0-1 valued
function on $\mathcal U$ such that $\phi(a,b)=1$ iff
$\models\phi(a,b)$. Then $\phi(x,y)$ has the strict order
property if and only if there are sequences
$(a_i,b_j:i,j<\omega)$ in $\mathcal U$ such that for each
$b\in{\mathcal U}$, the sequence $\{\phi(b,a_i)\}_i$ is
increasing -- therefore the pointwise limit
$\psi(x):=\lim_i\phi(x,a_i)$ is well-defined -- and
$\phi(b_j,a_j)<\phi(b_j,a_{j+1})$ for all $j<\omega$.

Now, suppose that the $\phi(x,a_i)$  are continuous functions on
$S_\phi({\mathcal U})$, the space of all complete $\phi$-types.
Suppose that $\phi$ has not the SOP, and
$\phi(x,a_i)\nearrow\psi(x)$. Then $\psi:S_\phi({\mathcal
U})\to\{0,1\}$ is continuous, because there is a $k$ such that
$\phi(x,a_k)=\phi(x,a_{k+1})=\cdots$. Conversely, suppose that
$\phi(x,a_i)\nearrow\psi(x)$ and $\psi$ is continuous. It is a
standard fact that an  increasing sequence of continuous functions
on a compact space which converges to a continuous function
converges uniformly (Dini's Theorem). Therefore, our sequence is
eventually constant, because the logic is 0-1 valued.

Therefore, it seems right to say that the SOP in classical logic
  is equivalent to the existence of a
pointwise convergent sequence (not necessary increasing) of
continuous functions such that its limit is not continuous. Our
next goal is to convince the reader that by a technical
consideration this is indeed the case.

\medskip
In functional analysis, a Banach space $X$ is called {\em weakly
sequentially complete} if every weak Cauchy sequence has a weak
limit. Similarly we define the following notion and
 will observe that this notion corresponds to
NSOP on the model-theoretic side.

\begin{dfn}  Let $X$ be a topological space and  $F\subseteq C(X)$. We say that $F$ has
the {\em weak sequential completeness property} (or short SCP) if
the limit of each pointwise convergent sequence $\{f_n\}\subseteq
F$ is continuous.
\end{dfn}

As we will see shortly, the following statement is a
generalization of a well known model theoretic fact, i.e. SOP
implies OP.

\begin{fct} \label{SOP->OP}
Let $X$ be a compact space and $F\subseteq C(X)$ a bounded subset.
If $F$ is relatively weakly compact in $C(X)$, then $F$ has the
SCP.
\end{fct}
\begin{proof}
Suppose that $F$ is relatively weakly compact in $C(X)$, and
$\{f_n\}$ is a sequence in $F$ which pointwise converges to $f$.
Since the pointwise topology and weak topology are the same (see
Fact~\ref{Grothendieck-lemma} above), so $f$ as a cluster point of
$\{f_n\}$ is continuous.
\end{proof}

 The next result  is
another application of the Eberlein-Grothendieck criterion:

\begin{thm} \label{nip+scp=stable}
Let $X$ be a compact space and $A\subseteq C(X)$ be bounded. Then
$A$ is  relatively weakly compact in $C(X)$ iff it has RSC and
SCP.
\end{thm}
\begin{proof} First we show that $cl_p(A)\subseteq C(X)$ if every sequence of
$A$ has a convergent subsequence in ${ \mathbb{R}}^X$ and the
limit of every convergent sequence of $A$ is continuous. Suppose
that $A$ has RSC and SCP. Let $\{f_n\}_n\subseteq A$ and
$\{a_m\}_m\subseteq X$, and suppose that the double limits
$\lim_m\lim_n f_n(a_m)$ and $\lim_n\lim_m f_n(a_m)$ exist. Let
$a$ be a cluster point of $\{a_m\}_m$. By RSC, there is a
convergent subsequence $f_{n_k}$ such that $f_{n_k}\to f$.
Therefore $\lim_m\lim_{n_k} f_{n_k}(a_m)=\lim_mf(a_m)$ and
$\lim_{n_k}\lim_m f_{n_k}(a_m)=\lim_{n_k}f_{n_k}(a)=f(a)$. By SCP,
$\lim_mf(a_m)=f(a)$. Since the double limits exist, it is easy to
verify that $\lim_m\lim_n f_{n}(a_m)=\lim_m\lim_{n_k}
f_{n_k}(a_m)$ and $\lim_n\lim_m f_{n}(a_m)=\lim_{n_k}\lim_m
f_{n_k}(a_m)$. So $A$ has the double limit property and thus it
is relatively weakly compact in $C(X)$. The converse follows from
Facts~\ref{IP->OP} and \ref{SOP->OP}.
\end{proof}

\begin{proposition} \label{SCP->NSOP}
If the set $\{\phi(x,a):a\in\mathcal{U}\}$ has the SCP, then
$\phi(x,y)$ is NSOP.
\end{proposition}
\begin{proof}  Suppose, for a contradiction, that  $\{\phi(x,a):a\in\mathcal{U}\}$ has the SCP
 and $\phi$ is SOP. By SOP, there are $(a_ib_i:i<\omega)$ in the monster model $\mathcal U$ such that $\phi({\mathcal U},a_i)\leqslant\phi({\mathcal
U},a_{i+1})$ and $\phi(b_j,a_i)<\phi(b_i,a_j)$ for all $i<j$. Let
$b$ be a cluster point of $\{b_i\}_{i<\omega}$. By SCP,
$\phi(S_\phi({\mathcal U}),a_i)\nearrow\psi$ and $\psi$ is
continuous. But
$\lim_i\lim_j\phi(b_j,a_i)=0<1=\lim_i\lim_j\phi(b_i,a_j)$ and by
continuity $\psi(b)<\psi(b)$, a contradiction.
\end{proof}

The following example shows that the converse does not hold in
analysis. It was suggested to us by  M\'{a}rton Elekes.

\begin{exa} \label{exa}
 Let $X$ be the Cantor set. Let $H=\{0\}\cup(X\cap(2/3,1))$. (We
note that $H$ is $\Delta_2^0$, i.e. it is $F_\sigma$ and
$G_\delta$ at the same time, but neither open nor closed.)  Then
it is easy to see that there exists a sequence $H_n$ of clopen
subsets of $X$ such that if $f_n$ is the characteristic function
of $H_n$ and $f$ is the characteristic function of $H$ then
$f_n\to f$ pointwise. Let $A =\{f_n :n<\omega\}$. Then all $f_n$
are continuous, uniformly bounded (even 0-1 valued), the
pointwise closure is $A\cup\{f\}$ (which are all Baire class 1
functions), and all monotone sequences in $A$ are eventually
constant: indeed, if there were a true monotone subsequence then
its limit would be the characteristic function of an open or a
closed set, but $H$ is neither open nor closed. Also, we note that
$A$ has the RSC  but it is not relatively weakly compact in
$C(X)$.
\end{exa}

Again we give a topological presentation of a model theoretic
property. For this, we need some definitions. Let $M$ be a
saturated enough structure and $\phi:M\times M\to\{0,1\}$ a
formula. For subsets $B,D\subseteq M$, we say that $\phi(x,y)$
has the {\em order property on } $B\times D$ (short OP on
$B\times D$) if there are   sequences $(a_i)\subseteq B$,
$(b_i)\subseteq D$ such that $\phi(a_i,b_j)$ holds if and only if
$i<j<\omega$. We will say that $\phi(x,y)$ has the {\em NIP on
$B\times D$}, if for the set $A=\{\phi(a,y):S_y(D)\to\{0,1\}~|a\in
B\}$, any of the cases in Lemma~\ref{equivalence} holds.

\begin{proposition} \label{NSOP=SCP}
Suppose that $T$ is a theory. Then the following are equivalent:
\begin{itemize}
             \item [{\em (i)}] $T$ is NSOP.
             \item [{\em (ii)}] For each indiscernible sequence $(a_n)_{n<\omega}$ and
formula $\phi(x,y)$, if the sequence $(\phi(x,a_n))_{n<\omega}$
pointwise converges on $S_\phi(\mathcal{U})$, then its limit is
continuous.
\end{itemize}
\end{proposition}
\begin{proof} (i) $\Rightarrow$ (ii): Suppose that there are an
indiscernible sequence $(a_n)_{n<\omega}$ and a formula
$\phi(x,y)$ such that the sequence $(\phi(x,a_n))_{n<\omega}$
pointwise converges but its limit is not continuous. Since the
limit is not continuous, $\tilde{\phi}(y,x)=\phi(x,y)$ has OP on
$\{a_n\}_{n<\omega}\times S_\phi({\mathcal U})$. Since every
sequence in $\{\phi(x,a_n)\}_{n<\omega}$ has a pointwise
convergent subsequence, $\tilde{\phi}(y,x)$ is NIP on
$\{a_n\}_{n<\omega}\times S_\phi({\mathcal U})$.  The following
argument is classic (see   \cite{Poi} and \cite{S}).  Since
$\tilde{\phi}(y,x)$ has OP, there  is a sequence $\{b_N\}\subseteq
S_\phi({\mathcal U})$ such that $\tilde\phi(a_i,b_N)$ holds if
and only if $i<N$.
 By NIP,
there is some integer $n$ and $\eta : n \rightarrow \{0,1\}$ such
that $\bigwedge_{i<n} \tilde\phi(a_i,x)^{\eta(i)}$ is
inconsistent. (Recall  that for a formula $\varphi$, we use the
notation $\varphi^0$ to mean $\neg\varphi$ and $\varphi^1$ to mean
$\varphi$.) Starting with that formula, we change one by one
instances of $\neg\tilde\phi(a_i,x) \wedge \tilde\phi(a_{i+1},x)$
to $\tilde\phi(a_i,x) \wedge \neg\tilde\phi(a_{i+1},x)$. Finally,
we arrive at a formula of the form $\bigwedge_{i<N}
\tilde\phi(a_i,x) \wedge \bigwedge_{N\leq i<n}
 \neg\tilde\phi(a_i,x)$. The tuple $b_N$ satisfies that formula.
 Therefore, there is some
$i_0<n$, $\eta_0 : n \rightarrow \{0,1\}$ such that
$$\bigwedge_{i\neq i_0, i_0+1} \tilde\phi(a_i,x)^{\eta_0(i)} \wedge \neg\tilde\phi(a_{i_0},x) \wedge \tilde\phi(a_{i_0+1},x)$$
is inconsistent, but
$$\bigwedge_{i\neq i_0, i_0+1} \tilde\phi(a_i,x)^{\eta_0(i)} \wedge \tilde\phi(a_{i_0},x) \wedge \neg\tilde\phi(a_{i_0+1},x)$$
is consistent. Let us define $\varphi(\bar a,x)=\bigwedge_{i\neq
i_0,i_0+1} \tilde\phi(a_i,x)^{\eta_0(i)}$.  Increase the sequence
$(a_i : i<\omega)$ to an indiscernible sequence $(a_i:i\in
\mathbb Q)$. Then for $i_0 \leq i<i' \leq i_0+1$, the formula
$\varphi(\bar a,x) \wedge \tilde\phi(a_i,x) \wedge
\neg\tilde\phi(a_{i'},x)$ is consistent, but $\varphi(\bar a,x)
\wedge \neg\tilde\phi(a_i,x) \wedge \tilde\phi(a_{i'},x)$ is
inconsistent. Thus the formula $\psi(x,y) = \varphi(\bar a,x)
\wedge \tilde\phi(y,x)$ has the strict order property.

(ii) $\Rightarrow$ (i): Suppose that the formula $\phi(x,y)$ has
SOP as witnessed by a sequence $(a_nb_n:n<\omega)$. Then the
formula $\psi(y_1,y_2)=\forall x(\phi(x,y_1)\to\phi(x,y_2))$
defines a continuous pre-order for which the sequence
$(a_n:n<\omega)$ forms an infinite chain. Replace
$(a_n)_{n<\omega}$ by an indiscernible sequence
$(c_n)_{n<\omega}$, and return to $\phi(x,y)$. Therefore,
$\phi(x,y)$ has SOP as witnessed by the sequence
$(c_nb_n:n<\omega)$. Now, $\phi(S_\phi({\mathcal
U}),c_n)\nearrow\varphi$ but $\varphi$ is not continuous.
\end{proof}

We now provide a proof of Shelah's theorem (\cite{Sh},
Theorem~4.1).

\begin{cor} \label{Shelah-continuous}  Suppose that $T$ is NIP and NSOP. Then $T$
is stable.
\end{cor}
\begin{proof} Let $\phi(x,y)$ be a formula, $(a_n)_{n<\omega}$ an
indiscernible sequence, and $(b_n)_{n<\omega}$ an arbitrary
sequence. Suppose that the double limits
$\lim_m\lim_n\phi(b_n,a_m)$ and $\lim_n\lim_m\phi(b_n,a_m)$ exist.
By NIP, there is a convergent subsequence $\phi(x,a_{m_k})$ such
that $\phi(x,a_{m_k})\to\psi(x)$ on $S_\phi(\mathcal{U})$.
Therefore, $\lim_n\lim_k\phi(b_n,a_{m_k})=\lim_n\psi(b_n)$ and
$\lim_k\lim_n\phi(b_n,a_{m_k})=\lim_k\phi(b,a_{m_k})=\psi(b)$
where $b$ is a cluster point of $\{b_n\}$. By NSOP,
$\lim_n\psi(b_n)=\psi(b)$. So the double limits are the same and
thus $T$ is stable. (Compare Theorem~\ref{nip+scp=stable}.)
\end{proof}

\subsection*{Theorems of Eberlein-\v{S}mulian and Shelah}
The well known compactness theorem of Eberlein and \v{S}mulian
says that relative compactness, relative sequential compactness
and relative countable compactness are equivalent for the weak
topology of a Banach space. Now, we show the correspondence
between Shelah's theorem and the Eberlein-\v{S}mulian theorem.

\begin{proposition} \label{Shelah=Eberlein}
Suppose that $X$ is a space of the form $S_\phi(M)$ and
$A=\{\phi(a,y):a\in M\}$ where $M$ is a sufficiently saturated
model of a theory $T$ and $\phi(x,y)$ a formula. Then the
following are equivalent.
\begin{itemize}
  \item [{\em (i)}] {\bf The Eberlein-\v{S}mulian theorem:}
  For every $A\subseteq C(X)$, the following statements are equivalent:
    \begin{itemize}
        \item [{\em (a)}]  The weak closure of $A$ is weakly compact in $C(X)$.
        \item [{\em (b)}] Each sequence of elements of $A$ has a subsequence that is weakly convergent in $C(X)$.
    \end{itemize}
 \item [{\em (ii)}] {\bf Shelah's theorem:} The following statements are  equivalent:
    \begin{itemize}
        \item [{\em (a$'$)}]  $T$ is stable.
        \item [{\em (b$'$)}]  $T$ has the NIP and the NSOP.
    \end{itemize}
\end{itemize}
\end{proposition}
\begin{proof} First, we note that by the Eberlein-Grothendieck criterion,
(a)~$\Leftrightarrow$~(a$'$).

It suffices to show that (b)~$\Leftrightarrow$~(b$'$). Suppose
that $(f_n)$ is a sequence of the form $(\phi(a_n,y))$ where
$(a_n)$ is an indiscernible sequence. By (b), there is a
subsequence $(f_{n_k})$ that is convergent. Therefore, $T$ has
NIP. Again by (b), its limit is continuous, so $T$ has NSOP, and
(b$'$) holds. Conversely, suppose that $T$ has NIP and NSOP. Let
$(f_n)$ be a sequence of the form $(\phi(a_n,y))$ where $(a_n)$
is an arbitrary sequence. By NIP, $(f_n)$ has a convergent
subsequence $(f_{n_k})$. Replace $(a_n)$ by an  indiscernible
sequence $(c_n)$. Then, by NSOP,  $f=\lim_kf_{n_k}$ is
continuous. So, (b) holds.
\end{proof}

To summarize:
$$\begin{array}{cccccc}
    \textrm{Logic:~~~~~~~~~} & \textrm{Stable} & \Longleftrightarrow & \textrm{NIP} & +  & \textrm{ NSOP} \\
    \textrm{ } &   &  &   &  &     \\
   \textrm{Analysis:~~~~~} & \textrm{Weakly Compact} & \Longleftrightarrow & \textrm{RSC} & + & \textrm{SCP}
\end{array}$$

\medskip
Of course, the Eberlein-\v{S}mulian theorem is proved for
arbitrary Banach spaces (even normed spaces), but it follows
easily from the case $C(X)$ (see \cite{Fremlin4},  Theorem~462D).
On the other hand, the above argument implicitly shows that
countable  compactness implies compactness.

Earlier  we defined angelic topological spaces. Roughly an angelic
space is one for which the conclusions of the Eberlein-\v{S}mulian
theorem hold. By the previous observations one can say that
`first order logic is angelic.'

\noindent\hrulefill



\begin{thebibliography}{EGGN07}
\bibitem[Ben14]{Ben-Gro}  I. Ben-Yaacov,  Model theoretic stability and definability of types, after A. Grothendiek, Bulletin of Symbolic Logic, (2014), 20, pp 491-496.
\bibitem[Ben13]{Ben2} I. Ben-Yaacov, On theories of random variables, Israel J. Math. 194 (2013), no. 2, 957-1012
\bibitem[BBHU08]{BBHU} I. Ben-Yaacov, A. Berenstein, C. W. Henson, A. Usvyatsov, \emph{Model theory for metric structures},
     Model theory with Applications to Algebra and Analysis, vol. 2
   (Z. Chatzidakis, D. Macpherson, A. Pillay, and A. Wilkie, eds.),
    London Math Society Lecture Note Series, vol. 350, Cambridge University Press, 2008.
\bibitem[BU10]{BU} I. Ben-Yaacov, A. Usvyatsov, Continuous first order logic and local stability, Transactions of the American Mathematical Society \textbf{362} (2010), no. 10, 5213-5259.
\bibitem[BFT78]{BFT} J. Bourgain, D. H. Fremlin, and M. Talagrand. Pointwise compact sets of baire-measurable functions. American Journal of Mathematics, 100(4):pp. 845–886, 1978.
\bibitem[Fol99]{Folland} G. B. Folland, {\em Real Analysis, Modern Techniquess and Their Appilcations}, 2nd ed. Wiley, New York (1999).
\bibitem[Fre75]{Fremlin75} D. H. Fremlin, Pointwise compact sets of measurable functions, Manuscripta Math. 15 (1975) 219- 242.
\bibitem[Fre04]{Fremlin3} D. H. Fremlin, {\em Measure Theory}, vol.3, (Measure Algebras, Torres Fremlin, Colchester, 2004).
\bibitem[Fre06]{Fremlin4} D. H. Fremlin, {\em Measure Theory}, vol.4, (Topological Measure Spaces, Torres Fremlin, Colchester, 2006).
\bibitem[GM14]{GM} E. Glasner, M. Megrelishvili, Eventual nonsensitivity and dynamical systems, arXiv 2014.
\bibitem[Gro52]{Gro} A. Grothendieck, Crit\`{e}res de compacit\'{e} dans les espaces fonctionnels g\'{e}n\'{e}raux, American Journal of Mathematics {\bf 74} ,168-186, (1952).
\bibitem[HP11]{HP}  E. Hrushovski, and A. Pillay, On NIP and Invariant measures, J. Eur. Math. Soc. 13, 1005–1061 (2011)
\bibitem[Iba14]{Iba14} T. Ibarluc\'{i}a, The dynamical hierachy for Roelcke precompact Polish groups, arXiv:1405.4613v1, 2014.
\bibitem[Kei87]{Keisler} H. J. Keisler, Measures and forking, Annals of Pure and Applied Logic 45 (1987), 119-169.
\bibitem[KN63]{KN}  J. L. Kelley, I. Namioka, {\em Linear topological spaces},  New York (1963).
\bibitem[Kha14]{K} K. Khanaki, {\em Amenability, extreme amenability,  model-theoretic stability, and NIP in integral logic}, ArXiv 2014
\bibitem[KP18]{KP} K. Khanaki, A. Pillay, \emph{Remarks on NIP in a model}, arXiv:1706.04674
\bibitem[P16]{Pillay-Grothendieck} A. Pillay, Generic stability and Grothendieck, South American Journal of Logic Vol. 2, n. 2,(2016), p. 1-6.
\bibitem[Pil96]{Pil} A. Pillay, Geometric stability theory, Oxford logic guides, Clarendon Press, 1996.
\bibitem[Poi00]{Poi} B. Poizat, A course in model theory: An introduction to contemporary mathematical logic, Springer, New York, (2000).
\bibitem[Ros74]{Ros} H. P. Rosenthal, {\em A characterization of Banach spaces containing $l^1$}, Proc. Nat. Acad. Sci. U.S.A. 71 (1974), 2411-2413.
\bibitem[Ros77]{Ros2} H. P. Rosenthal. {\em Pointwise compact subsets of the first baire class}. Amer. Jour. of Math., 99(2):362–378, 1977.
\bibitem[Sim14a]{S} P. Simon, A guide to NIP theories, lecture note (2014).
\bibitem[Sim14b]{S2} P. Simon, Rosenthal compacta and NIP formulas, arXiv 2014
\bibitem[She71]{Sh} S. Shelah, Stability, the f.c.p., and superstability; model theoretic properties of formulas in first order theory, Annals of Mathematical Logic, vol. 3 (1971), no. 3, pp. 271-362.
\bibitem[SF93]{SF} S. Shelah, D. H. Fremlin, Pointwise compact and stable sets of measurable functions, J. Symbolic Logic 58 (1993) 435-455.
\bibitem[Tal87]{Talagrand} M. Talagrand, The Glivenko-Cantelli problem, {\em Ann. of Probability} 15 (1987) 837-870.
\end{thebibliography}
\end{document}